\documentclass[runningheads]{llncs}
\usepackage[T1]{fontenc}
\usepackage{hyperref}       % hyperlinks
\usepackage{url}            % simple URL typesetting
\usepackage{booktabs}       % professional-quality tables
\usepackage{amsfonts}       % blackboard math symbols
\usepackage{tablefootnote}
\usepackage{verbatim}
\usepackage{nicefrac}       % compact symbols for 1/2, etc.
\usepackage{microtype}      % microtypography
\usepackage{lipsum}
\usepackage{mathtools}

\usepackage{amsmath,amssymb}
\usepackage{algorithm,algorithmic}
\usepackage{pifont}
\usepackage{cases}
\usepackage{subcaption,graphicx}
\usepackage{stackengine}    % circled symbols
\usepackage{wrapfig}
\usepackage{enumitem}
\usepackage{cleveref}

\newtheorem{assumption}[theorem]{Assumption}

\newcommand{\eps}{\varepsilon}

\newcommand{\one}{\mathbf{1}}

\renewcommand{\hat}{\widehat}

\DeclareMathOperator*{\argmin}{arg\,min}

\DeclareMathOperator*{\Argmin}{Arg\,min}

\DeclareMathOperator{\image}{Im}
\DeclareMathOperator{\prox}{prox}
\DeclareMathOperator{\proxv}{proxv}

\DeclareMathOperator{\col}{col}
\DeclareMathOperator{\diag}{diag}
\DeclareMathOperator{\sign}{sign}

\newcommand{\R}{\mathbb{R}}

\newcommand{\mA}{{\bf A}}

\newcommand{\mI}{{\bf I}}

\newcommand{\mQ}{{\bf Q}}

\newcommand{\mW}{{\bf W}}

\newcommand{\cE}{{\mathcal{E}}}

\newcommand{\cG}{{\mathcal{G}}}

\newcommand{\cU}{{\mathcal{U}}}
\newcommand{\cV}{{\mathcal{V}}}

\newcommand{\bb}{{\bf b}}

\newcommand{\bu}{{\bf u}}

\newcommand{\bx}{{\bf x}}
\newcommand{\by}{{\bf y}}
\newcommand{\bz}{{\bf z}}

\newcommand{\norm}[1]{\left\| #1 \right\|}

\newcommand{\angles}[1]{\left\langle #1 \right\rangle}
\newcommand{\cbraces}[1]{\left( #1 \right)}
\newcommand{\sbraces}[1]{\left[ #1 \right]}
\newcommand{\braces}[1]{\left\{ #1 \right\}}
\def\<#1,#2>{\langle #1,#2\rangle}

\newcommand{\lmax}{\lambda_{\max}}

\newcommand{\smax}{\sigma_{\max}}
\newcommand{\sminp}{\sigma_{\min^+}}
\newcommand{\lminp}{\lambda_{\min^+}}

\usepackage[colorinlistoftodos,bordercolor=blue,backgroundcolor=blue!20,linecolor=blue,textsize=scriptsize]{todonotes}

\title{Dual Smoothing for Decentralized Optimization}
\date{}

\author{
	Alexander Rogozin\orcidID{0000-0003-3435-2680} \and
	Nhat Trung Nguyen\orcidID{} \and
	Hamed Azami Zenuzagh\orcidID{} \and
	Alexander Gasnikov\orcidID{0000-0002-7386-039X}
}
\authorrunning{A. Rogozin et al.}
\institute{Moscow Institute of Physics and Technology}

\begin{document}
% \linenumbers
\maketitle

\begin{abstract}

Decentralized optimization is widely used in different fields of study such as distributed learning, signal processing, and various distributed control problems. In these types of problems, nodes of the network are connected to each other and seek to optimize some objective function. In this article, we present a method for smoothing the non-smooth and non-strongly convex problems. This is done using the dual smoothing technique. We study two types of problems: consensus optimization of linear models and coupled constraints optimization. It is shown that these two problem classes are dual to each other.

\end{abstract}

%%================================%%

% \keywords{ convex optimization, distributed optimization, Decentralized Optimization}



\section{Introduction}

Decentralized optimization is used in multi-agent systems control \cite{ren2006consensus,ren2008distributed}, large-scale data processing \cite{konecny2016federated}, power systems control \cite{gan2012optimal,ram2009distributed}. Is it assumed that no centralized server is present and therefore the nodes communicate only to their immediate neighbors in the network, which is represented by a connected and undirected graph $\cG = (\cV, \cE)$. Each node locally holds an objective function and can perform local computations. We study two types of problems given below.
\begin{align}\tag{Con}\label{prob:consensus_intro}
	\begin{aligned}
	\min_{x_1\in Q, \ldots, x_n\in Q}~ &\sum_{i=1}^n f_i(A_i x_i - b_i) \\
	\text{s.t. } &x_1 = \ldots = x_n
	\end{aligned}
\end{align}
\begin{align}\tag{Coupl}\label{prob:coupled_intro}
	\begin{aligned}
		\min_{x_1\in Q_1, \ldots, x_n\in Q_n}~ &\sum_{i=1}^n f_i(x_i) \\
		\text{s.t. } &\sum_{i=1}^n (A_i x_i - b_i) = 0
	\end{aligned}
\end{align}
%\begin{minipage}{0.5\textwidth}
%	\begin{align}\tag{Con}\label{prob:consensus_intro}
%		\begin{aligned}
%		\min_{x_1\in Q, \ldots, x_n\in Q}~ &\sum_{i=1}^n f_i(A_i x_i - b_i) \\
%		\text{s.t. } &x_1 = \ldots = x_n
%		\end{aligned}
%	\end{align}
%\end{minipage}
%\begin{minipage}{0.5\textwidth}
%	\begin{align}\label{prob:coupled_intro}
%%		\begin{aligned}
%		\min_{x_1\in Q_1, \ldots, x_n\in Q_n}~ &\sum_{i=1}^n f_i(x_i) \tag{Coupl} \\
%		\text{s.t. } &\sum_{i=1}^n (A_i x_i - b_i) = 0 \nonumber
%	%	\end{aligned}
%	\end{align}
%\end{minipage}
Here $f_i$ are convex functions locally held by the nodes, $x_i$ are local variables and $Q, Q_1, \ldots, Q_n$ are closed convex sets. Problem \eqref{prob:consensus_intro} is consensus optimization of linear models and \eqref{prob:coupled_intro} is optimization with coupled constraints. In \eqref{prob:consensus_intro}, each agent locally stores data $(A_i, b_i)$ and a loss function $f_i$. The nodes collectively solve the optimization problem while keeping their local trajectories close to consensus, i.e. the constraint $x_1 = \ldots = x_n$. In problem \eqref{prob:coupled_intro}, the agents' local variables are tied by an affine constraint, but the constraint itself is stored in a distributed way between the nodes. Problem \eqref{prob:consensus_intro} can be stated as a special case of \eqref{prob:coupled_intro} if matrices $A_i$ are taken as slices of the graph Laplacian matrix \cite{yarmoshik2024coupled}. Moreover, problems \eqref{prob:consensus_intro} and \eqref{prob:coupled_intro} are dual to each other up to redefinition of $f_i$ and transposition of $A_i$. In this work we focus on convex nonsmooth objectives. Let us discuss several examples of problems \eqref{prob:consensus_intro} and \eqref{prob:coupled_intro}.

\vspace{0.3cm}
\noindent\textbf{Decentralized mean absolute error optimization}. Consider a special case of \eqref{prob:consensus_intro} with $f_i(y_i) = \norm{y_i}_1$.
\begin{align*}
	\min_{x_1\in Q, \ldots, x_n\in Q}~ &\sum_{i=1}^n \norm{A_i x_i - b_i}_1 \\
	\text{s.t. } &x_1 = \ldots = x_n
\end{align*}
Set $Q$ may be either a unit simplex or the whole space $\R^d$, where $d$ is the dimension of $x_1, \ldots, x_n$. Moreover, we may use Huber smoothing instead of 1-norm, i.e. set
\begin{align*}
	f_i(y) = [g(y_1) \ldots g(y_n)]^\top, \text{ where } g(y_i) =
	\begin{cases}
		\frac{y_i^2}{2},~ |y_i| < 1, \\ |y_i| - \frac12,~ |y_i|\geq 1.
	\end{cases}
\end{align*}

\vspace{0.3cm}
\noindent\textbf{Decentralized basis pursuit}. This problem is seeking the solution of a linear system with the least 1-norm. It can be seen as a special case of \eqref{prob:coupled_intro} with $f_i(y_i) = \norm{y_i}_1$.
\begin{align*}
	\min_{x_1\in Q_1, \ldots, x_n\in Q_n}~ &\sum_{i=1}^n \norm{x_i}_1 \\
	\text{s.t. } &\sum_{i=1}^n (A_i x_i - b_i) = 0
\end{align*}
Of our interest is only the case $Q_ 1 = \ldots = Q_n = \R^d$.

\vspace{0.3cm}
\noindent\textbf{Related work}. Decentralized consensus optimization (i.e. problem \eqref{prob:consensus_intro} with $A_i$ identity matrices) can be called a theoretically well-studied field for synchronous first-order methods (i.e. methods requiring only gradient information and performing synchronous communication rounds). Decentralized methods that we consider operate two types of steps -- (synchronous) communication rounds and local gradient updates. The number of communication rounds required to achieve the given accuracy is called communication complexity, and the corresponding number of gradient steps is named computational complexity. For $\mu$-strongly convex objectives with $L$-Lipschitz gradient and a static network with condition number $\chi$, primal \cite{kovalev2020optimal} and dual \cite{scaman2017optimal} methods are known that achieve communication complexity $O\cbraces{\sqrt{\chi L/\mu}\ln(1/\eps)}$, which is also shown to match the lower complexity bound \cite{scaman2017optimal}. For time-varying graphs with worst-case condition number $\chi$, methods requiring $O(\chi\sqrt{L/\mu}\ln(1/\eps))$ communication rounds are also known \cite{kovalev2021lower,kovalev2021adom,li2021accelerated}. These algorithms are also theoretically optimal, i.e. matching the lower complexity bound derived in \cite{kovalev2021lower}. For nonsmooth convex problems, where the gradient norm is uniformly bounded by $M$, a penalty method with sliding technique \cite{lan2016gradient} was applied to achieve communication complexity $O(\sqrt\chi MR/\eps)$ \cite{scaman2018optimal,uribe2017optimal}, which is an optimal complexity. For time-varying networks, optimal algorithms were proposed with complexity $O(\chi MR/\eps)$ for convex setup and $O(\chi M/\sqrt{\mu\eps})$ for $\mu$-strongly convex setup.

Coupled constraints optimization, i.e. problem \eqref{prob:coupled_intro}, is not studied as extensively as consensus optimization. Prior to first-order methods, coupled constraints problem was solved by ADMM approaches \cite{chang2016proximal,falsone2020tracking,li2018accelerated,wu2022distributed}. Gradient methods were used in \cite{doan2017distributed}, and in \cite{yarmoshik2024coupled} a first-order method optimal for $L$-smooth $\mu$-strongly convex objectives was proposed. It has communication complexity $O(\sqrt{\chi\kappa_A L/\mu}\ln(1/\eps))$, where $\kappa_A$ is related to condition numbers of the constraint matrices.

\vspace{0.3cm}
\noindent\textbf{Paper contribution}. In this work, we show how dual smoothing technique \cite{nesterov2005smooth} can be applied to consensus optimization of linear models \eqref{prob:consensus_intro} and to coupled constraints optimization \eqref{prob:coupled_intro}. The idea of dual smoothing is based on the fact that a Fenchel conjugate of a $\mu$-strongly convex function is $1/\mu$-smooth \cite{kakade2009duality}. Therefore, we can regularize a problem, take its dual and get a smooth problem. We show the transition between problems \eqref{prob:consensus_intro} and \eqref{prob:coupled_intro} and their regularized versions via duality.

\vspace{0.3cm}
\noindent\textbf{Notation}. We let $\mathbb{R}^d$ the Euclidean $d$-dimension space. We let $\col(x_1, \ldots, x_n) = (x_1^\top \ldots x_n^\top)^\top$ be a column-stacked vector. We denote by $m$ the dimension of the stacked vector, i.e. $m = nd$ or $m = d_1 + \ldots + d_n$ depending on the context. Maximal eigenvalue and singular value of matrix $C$ are denoted as $\lmax(C)$ and $\smax(C)$, respectively. Minimal nonzero eigenvalue and singular value of $M$ are denoted $\lminp(C), \sminp(C)$, respectively. The network is represented by a connected undirected graph $\cG = (\cV, \cE)$ with $|\cV| = n$ nodes and $|\cE| = \ell$ edges. Given a closed convex function $h$, closed convex set $Q$ and a scalar $\lambda > 0$, prox operator is defined as $\prox_{\lambda h}^Q(x) = \argmin_{y\in Q}\sbraces{\lambda h(y) + \frac12\norm{y - x}_2^2}$. We also introduce <<prox value>> as $\proxv_{\lambda h}^Q(x) = \min_{y\in Q}\sbraces{\lambda h(y) + \frac12\norm{y - x}_2^2}$. A Kronecker product of matrices $B$ and $C$ is denoted $B\otimes C$.

 Set $\mQ$ is defined either as $\mQ = Q_1\times \ldots \times Q_n$ or $\mQ = Q^n$ depending on the context. 
 %A unit simplex in $\mathbb{R}^d$ is defined as $\Delta_d = \{x \in \mathbb{R}^d : \sum_{i=1}^{m} x_i = 1, x_i \geq 0, i = 1, \dots, m\}$ and $\Delta_d^m$ is a Cartesian product of $m$ simplices and a subset of $\mathbb{R}^{md}$. 
 For a vector $x\in\R^d$, $\norm{x}_p$ denotes its $p$-norm for $p\geq 0$. After that, $\one_n$ is a vector of length $n$ such that all of its components are one and $e_i$ is the $i$-th coordinate vector the dimension of which is clear from the context. %Uppercase letters represent matrices of dimension $n \times d$ with real components. For an arbitrary matrix $A \in \mathbb{R}^{n\times d}$, the entry in row $i$ and column $j$ is denoted by $[A]_{ij}$. An identity matrix of size $n \times n$ is denoted $\mathbf{I}_n$. We also introduce $\mathbf{A} = \col[A_1, A_2, \dots, A_m] = [A_1^\top, A_2^\top, \dots, A_m^\top]^\top \in \mathbb{R}^{mn \times d}$. For a vector $x \in \mathbb{R}^d$, $\log x$ is defined component-wise, i.e. $\log x = [\log x_1, \log x_2, \dots, \log x_d]^\top$. Let $\lambda_{\max}(\cdot)$ and $\lambda_{\min}^+(\cdot)$ denote maximum and minimum nonzero eigenvalues of a matrix, respectively. Also let $\sigma_{\max}(\cdot)$ and $\sigma_{\min}(\cdot)$ denote maximum and minimum nonzero singular values of a matrix.

Also let $\text{dist}(x, S)$ denote the Euclidean distance from point $x$ to set $S$. Let $B_\infty(x, R) = \braces{y:~ \norm{y - x}_\infty\leq R}$ denote the ball in sup-norm with center at $x$ and radius $R$. The indicator function for set $S$ is defined as $\mathbb{I}(x) = \braces{0 \text{ if } x\in S;~ +\infty \text{ if } x\notin S}$.

Generally, the vectors are denoted in lower case and matrices in upper case. Unless otherwise stated, we write $\bx = \col(x_1, \ldots, x_n)$, $\bb = \col(b_1, \ldots, b_n)$ and $\mA = \diag(A_1, \ldots, A_n)$.

\vspace{0.3cm}
\noindent\textbf{Paper organization}

In \Cref{sec:problem_statement} we introduce the problems of interest and deduce their duals. We continue in \Cref{sec:algorithms} with methods and complexity bounds. Concluding remarks are provided in \Cref{sec:conclusion}.

\section{Problem statement and assumptions}\label{sec:problem_statement}

Firstly, we make a standard convexity assumption for optimization and recall the basic definitions.
\begin{definition}
	Let $Q\subseteq\R^d$ be a closed convex set. Function $h: Q\to\R$ is $\mu$-strongly convex, where $\mu > 0$, if for any $x, y\in Q$ it holds
	\begin{align*}
		h(y)\geq h(x) + \angles{\nabla h(x), y - x} + \frac{\mu}{2}\norm{y - x}_2^2.
	\end{align*}
	If $\mu$ is put $\mu = 0$ in the equation above, the function is called just convex.
\end{definition}
\begin{definition}
	A function $h: Q\to\R^n$ is called $L$-smooth for some $L\geq 0$ if for any $x, y\in Q$ it holds
	\begin{align*}
		h(y)\leq h(x) + \angles{\nabla h(x), y - x} + \frac L2 \norm{y - x}_2^2.
	\end{align*}
\end{definition}
\begin{assumption}\label{assum:f_i_convex}
	Functions $f_i$ are convex.
\end{assumption}
We also introduce a gossip matrix that is widely used in decentralized optimization.
\begin{assumption}\label{assum:gossip_matrix}
	Gossip matrix $W\in\R^{n\times n}$ associated with graph $\cG$ has the following properties.\\
	1. $W$ is symmetric positive semi-definite. \\
	2. $W$ is network compatible, i.e. $W_{ij} = 0$ if $(i, j)\notin\cE$.\\
	3. $W x = 0$ if and only if $x_1 = \ldots = x_n$.
\end{assumption}
An example of gossip matrix is the graph Laplacian defined as $L = D - A$, where $D$ is a diagonal matrix holding node degrees and $A$ is the graph adjacency matrix. To equivalently formulate consensus constraints for vectors, we use $\mW = W\otimes \mI_d$. For vector $\bx = \col(x_1, \ldots, x_n)$, where $x_i\in\R^d,~ i = 1, \ldots, n$ we have $W\bx = 0$ if and only if $x_1 = \ldots = x_n$.

After that, we introduce spectral properties for matrices $W$ and $A_1, \ldots, A_n$. For gossip matrix $W$, we let
\begin{align}\label{eq:def_kappa_w}
	L_W = \lmax(W),~ \mu_W = \lminp(W),~ \kappa_W = \frac{L_W}{\mu_W}.
\end{align}
For matrices $A_1, \ldots, A_n$ we define
\begin{align}\label{eq:def_kappa_a}
	L_A = \max_{i=1,\ldots,n} \lmax(A_i A_i^\top),~~ \mu_A = \lminp\cbraces{\frac1n\sum_{i=1}^n A_i A_i^\top},~~ \kappa_A = \frac{L_A}{\mu_A}.
\end{align}
Note that $\mu_A$may differ from minimal positive eigenvalues of each $A_i$ separately.

%\begin{minipage}{0.4\textwidth}
%	\begin{align}\tag{Con'}\label{prob:consensus}
%			\begin{aligned}
%				\min_{\bx\in\mQ}~ &F(\by) \\
%				\text{s.t. } &\mW\bx = 0 \\
%				&\by = \mA\bx - \bb
%				\end{aligned}
%		\end{align}
%\end{minipage}
%\hspace{0.15\textwidth}
%\begin{minipage}{0.4\textwidth}
%	\begin{align}\tag{Coupl'}\label{prob:coupled}
%			\begin{aligned}
%				\min_{\bx\in\mQ}~ &F(\bx) \\
%				\text{s.t. } &\mA\bx+ \mW\by = \bb
%			\end{aligned}
%		\end{align}
%\end{minipage}

%After that, we introduce regularizer $G(\bx)$.
%\begin{assumption}
%	Regularizer $G(\bx)$ is $1$-strongly convex function on $\mQ$ separable in $x_1, \ldots, x_n$, i.e.
%	\begin{align*}
%		G(\bx) = \sum_{i=1}^n g_i(x_i) ~~\text{ and }~~ G(\by)\geq G(\bx) + \angles{\nabla G(\bx), \by - \bx} + \frac12 \norm{\by - \bx}_2^2.
%	\end{align*}
%\end{assumption}
Let each node hold regularizers $g_i$ and $h_i$ along with objective function $f_i$. Introduce $F(\bx) = \sum_{i=1}^n f_i(x_i),~ G(\bx) = \sum_{i=1}^n g_i(x_i),~ H(\bx) = \sum_{i=1}^n h_i(x_i)$. We now define regularized versions for consensus and coupled constraints problems. Given a gossip matrix, we rewrite problems \eqref{prob:consensus_intro} and \eqref{prob:coupled_intro} the following way.

\begin{minipage}{0.43\textwidth}
	\begin{align}\tag{ConR}\label{prob:consensus_reg}
		\hspace{-0.17\textwidth}\begin{aligned}
			\min_{\bx\in\mQ}~ &F(\by) + \lambda G(\by) + \mu H(\bx) \\
			\text{s.t. } &\mW\bx = 0,~ \by = \mA\bx - \bb \\
		\end{aligned}
	\end{align}
\end{minipage}
\hspace{0.03\textwidth}
\begin{minipage}{0.48\textwidth}
	\begin{align}\tag{CouplR}\label{prob:coupled_reg}
		\hspace{-0.07\textwidth}
		\begin{aligned}
			\min_{\bx\in\mQ}~ &F(\bx) + \lambda G(\bx) + \mu H(\by) \\
			\text{s.t. } &\mA\bx+ \mW\by = \bb
		\end{aligned}
	\end{align}
\end{minipage}

\vspace{0.3cm}
We now pass to dual problem formulation. Let us introduce a conjugate function w.r.t. a regularizer.
\begin{definition}
	For given convex functions $\varphi$ and $\psi$, closed convex set $S$ and scalar $\gamma > 0$, we define 
	$$\varphi_{\gamma\psi,S}^*(v) = \max_{u\in S} \cbraces{\angles{u, v} - \varphi(u) - \frac\gamma2\psi(u)}.$$
	If $\psi\equiv 0$, we write
	\begin{align*}
		\varphi_{S}^*(v) = \max_{u\in S} \cbraces{\angles{u, v} - \varphi(u)}.
	\end{align*}
\end{definition}

\begin{lemma}\label{lem:dual_problems} Dual problems to \eqref{prob:consensus_reg}, \eqref{prob:coupled_reg_dual} write as

\begin{minipage}{0.42\textwidth}
	\begin{align}\tag{ConRD}\label{prob:consensus_reg_dual}
		\hspace{-0.17\textwidth}
		\begin{aligned}
			\min_{\bz, \bu}~ &F^*_{\lambda G,\R^m}(\bz) + \angles{\bz, \bb} \\
			\text{s.t. } &\mW\bu+ \mA^\top\bz = 0
		\end{aligned}
	\end{align}
\end{minipage}
\hspace{0.05\textwidth}
\begin{minipage}{0.45\textwidth}
	\begin{align}\tag{CouplRD}\label{prob:coupled_reg_dual}
		\hspace{-0.17\textwidth}\begin{aligned}
			\min_{\bz}~ &F^*_{\lambda G,\mQ}(\mA^\top \bz) - \angles{\bz, \bb} \\
			\text{s.t. } &\mW\bz = 0
		\end{aligned}
	\end{align}
\end{minipage}
\end{lemma}
\begin{proof}
	Let us deduce a dual function to \eqref{prob:consensus_reg} up to a sign.
	\begin{align*}
		-&\min_{\bx\in\mQ, \by} [F(\bx) + \lambda G(\bx) - \angles{\bz, \mA\bx+ \mW\by - \bb}] \\
		&= \max_{\bx\in\mQ} [\angles{\mA^\top\bz, \bx} - F(\bx) - \lambda G(\bx)] + \max_{\by}[\angles{\mW\bz, \by}] - \angles{\bz, \bb}\\
		&= F_{\lambda G, \mQ}^* (\mA^\top\bz) - \angles{\bz, \bb} + \mathbb{I}(\mW\bz = 0).
	\end{align*}
	Analogously for \eqref{prob:coupled_reg} we get
	\begin{align*}
		-&\min_{\bx\in\mQ, y} [F(\by) + \lambda G(\bx) - \angles{\bz, \by - \mA\bx + \bb} + \angles{\bu, \mW\bx}] \\
		&=\max_{\by}[\angles{\bz, \by} - F(\by) - \lambda G(\by)] + \max_{\bx\in\mQ}[\angles{\bx, -\mW\bu - \mA^\top\bz}] + \angles{\bz, \bb} \\
		&=F_{\lambda G, \R^m}^*(\bz) + \angles{\bz, \bb} + \mathbb{I}(-\mW\bu - \mA^\top\bz = 0)
	\end{align*}
\end{proof}
Lemma \ref{lem:dual_problems} shows that consensus optimization and coupled constraints optimization problems are dual to each other up to the usage of $\mA^\top$ instead of $\mA$ and factors $\angles{\bz, \bb}$ and $\bb$. %Note that if we take $H\equiv 0$, problems \eqref{prob:consensus_reg_dual} and \eqref{prob:coupled_reg_dual} come down to constrained problems

%\begin{minipage}{0.45\textwidth}
%\begin{align}\tag{ConRD'}\label{prob:consensus_reg_dual_constrained}
%	\begin{aligned}
%		\min_{\bz, \bu}~ &F^*_{\lambda G,\R^m}(\bz) + \angles{\bz, \bb} \\
%		\text{s.t. } &\mW\bu + \mA^\top\bz = 0
%	\end{aligned}
%\end{align}
%\end{minipage}
%\hspace{0.01\textwidth}
%\begin{minipage}{0.50\textwidth}
%\begin{align}\tag{CouplRD'}\label{prob:coupled_reg_dual_constrained}
%	\begin{aligned}
%		\min_{\bz}~ &F^*_{\lambda G,\mQ}(\mA^\top \bz) - \angles{\bz, \bb} \\
%		\text{s.t. } &\mW\bz = 0
%	\end{aligned}
%\end{align}
%\end{minipage}

\section{Algorithms and complexity}\label{sec:algorithms}

\subsection{Preliminaries and base algorithm}

Decentralized optimization algorithms that we will use are based on optimization methods for affinely constrained problems. We recall the basic method APAPC \cite{kovalev2020optimal,salim2022optimal} which is a state-of-the-art method for this class of problems. We are interested in problem statement
\begin{align}\label{prob:general_affine_constraints}
	\min_{u\in\cU} P(u) \qquad \text{ s.t. } K u = c.
\end{align}

\begin{algorithm}[H]
	\caption{APAPC}
	\label{alg:apapc}
	\begin{algorithmic}[1]
		\STATE {\bf Parameters:}  %$x^0 \in \sX$, 
		$u^0 \in \cU$
		$\eta,\theta,\alpha>0$, $\tau \in (0,1)$
		%\State Set $y^0 = z^0 = 0 \in (\sY)^n$, $u_f^0 = u^0 = (x^0, y^0)$.
		\STATE Set $u_f^0 = u^0$, $z^0 = 0 \in \cU$
		\FOR{$k=0,1,2,\ldots$}{}
		\STATE $u_g^k = \tau u^k +  (1-\tau)u_f^k$\label{alg:apapc:line:x:1}
		\STATE $u^{k+\frac{1}{2}} = (1+\eta\alpha)^{-1}(u^k - \eta (\nabla P(u_g^k) - \alpha  u_g^k + z^k))$\label{alg:apapc:line:x:2}
		\STATE $z^{k+1} = z^k  + \theta K^\top (K u^{k+\frac{1}{2}} - c)$ \label{alg:apapc:line:z}
		\STATE $u^{k+1} = (1+\eta\alpha)^{-1}(u^k - \eta (\nabla P(u_g^k) - \alpha  u_g^k + z^{k+1}))$\label{alg:apapc:line:x:3}
		\STATE $u_f^{k+1} = u_g^k + \tfrac{2\tau}{2-\tau}(u^{k+1} - u^k)$\label{alg:apapc:line:x:4}
		\ENDFOR
	\end{algorithmic}
\end{algorithm}

Algorithm \ref{alg:apapc} has an optimal linear convergence rate, which we recall below.
\begin{proposition}[\cite{salim2022optimal}, Proposition 1]\label{prop:apapc}
	Assume that $c \in \image K$ and put\\ $\kappa_K = \frac{\lmax(K^\top K)}{\lminp(K^\top K)}$.
	Also assume that the function $P$ is $L_P$-smooth and $\mu_P$-strongly convex.
	Set the parameter values of \ref{alg:apapc} as $\tau = \min\braces{1, \frac12\sqrt\frac{\kappa_K}{\kappa_P}}$, $\eta= \frac1{4\tau L_P}$, $\theta = \frac1{\eta L_K}$ and $\alpha = \mu_P$ and let $u^*$ be the solution of \eqref{prob:general_affine_constraints}. Then in order to yield $u^N$ such that $\norm{u^N - u^*}\leq\eps$, Algorithm \ref{alg:apapc} requires $O(\sqrt{\kappa_P \kappa_K} \ln(1/\eps))$ communication rounds.
\end{proposition}

\Cref{alg:apapc} was applied to consensus optimization in \cite{kovalev2020optimal} and to coupled constraints optimization in \cite{yarmoshik2024coupled}.

\Cref{prop:apapc} is formulated for strongly convex smooth objectives. In order to use this result for non-strongly convex functions, we use a regularization technique.

\begin{lemma} \label{lem:regularization_affine}
	Let $h: \cU \to \mathbb{R}$ be convex and $L$-smooth function and introduce regularized function $h_{\mu}(x) = h(x) + \frac{\mu}{2} \|x^0 - x\|^2$. Suppose that there exists solution $x^*\in \Argmin\limits_{x: Kx = c} h(x)$ and $x_{\mu}^* =  \argmin\limits_{x: Kx = c} h_{\mu}(x)$. Define $M_h = h(x^*) - \min\limits_{x} h(x)$. Assume that $\|x^0 - x^*\|^2 \leq R^2$. If regularization parameter is set to $\mu = \eps/R^2$ and regularized problem is solved up to accuracy $\delta = O(\eps^2)$, i.e. some method yields $\hat x$ such that $\norm{\hat x - x_\mu^*}_2^2\leq \delta$, then $h(\hat x) - h(x^*)\leq \eps$.
\end{lemma}

\begin{proof}
	We have
	\begin{align*}
		h&(\hat x) - h(x^*) \\ 
		&\leq h_{\mu}(\hat x) - h_{\mu}(x_{\mu}^*) + h_{\mu}(x_{\mu}^*) - h(x^*) \\ 
		&\leq h_{\mu}(\hat x) - h_{\mu}(x_{\mu}^*) + h_{\mu}(x^*) - h(x^*) \\
		&= h_{\mu}(\hat x) - h_{\mu}(x_{\mu}^*) + \frac{\mu}{2} \|x^0-x^*\|^2 \\
		&\leq \langle \nabla h_\mu(x_\mu^*), x - x_\mu^* \rangle + \frac{L+\mu}{2} \|x-x_\mu^*\|^2 + \frac{\mu}{2} \|x^0-x^*\|^2 \\
		&\leq \| \nabla h_{\mu} (x^*_{\mu}) \| \cdot \|x - x_\mu^*\| + \frac{L+\mu}{2} \|x-x_\mu^*\|^2 + \frac{\mu}{2} \|x^0-x^*\|^2 \\
		&\leq\sqrt{ 2(L+\mu) \left(h_\mu(x^*_\mu) - \min_{x} {h_\mu(x)}\right) \|x - x_\mu^*\|^2} \\
		&\qquad + \frac{L+\mu}{2} \|x-x_\mu^*\|^2 + \frac{\mu}{2} \|x^0-x^*\|^2 \\
		&\leq\sqrt{ 2(L+\mu) \left(h_\mu(x^*) - \min_{x} {h(x)}\right) \|x - x_\mu^*\|^2} \\
		&\qquad + \frac{L+\mu}{2} \|x-x_\mu^*\|^2 + \frac{\mu}{2} \|x^0-x^*\|^2 \\
		&\leq\sqrt{ 2(L+\mu) \left(h(x^*) - \min_{x} h(x) + \frac{\mu}{2} \|x^0-x^*\|^2\right) \|x - x_\mu^*\|^2} \\
		&\qquad + \frac{L+\mu}{2} \|x-x_\mu^*\|^2 + \frac{\mu}{2} \|x^0-x^*\|^2 \\
		&\leq\sqrt{ 2 (L+\mu) \left(M + \frac{\mu R^2}{2}\right)\delta} + \frac{(L+\mu)\delta}{2} + \frac{\mu R^2}{2} \\
		&\leq\sqrt{ 2 (L+\mu) \left(M + \frac{\varepsilon}{2}\right)\delta} + \frac{(L+\mu)\delta}{2} + \frac{\mu R^2}{2} \\
		&\leq \frac{\varepsilon}{64} + \frac{\varepsilon}{4} + \frac{\varepsilon}{2} 
		< \varepsilon
	\end{align*}
\end{proof}
In particular, for non-strongly convex objectives problem \eqref{prob:general_affine_constraints} can be solved using regularization.
\begin{corollary}
	Let assumptions of \Cref{prop:apapc} hold except strong convexity. Then solving regularized problem $\min_{u\in\cU}~ h(u) + \mu\norm{u}_2^2/2~~ \text{s.t. } Ku = c$ requires $O(\sqrt{\kappa_K L_G R^2/\eps} \ln(1/\eps))$ iterations of \Cref{alg:apapc} and yields $\hat u$ such that $\norm{\hat u - u^*}_2^2\leq \eps$.
\end{corollary}

Apart from convergence result for affinely constrained optimization, we recall a property for smoothness of Fenchel conjugate functions.
\begin{proposition}\label{prop:conjugate_smoothness}\cite{kakade2009duality}
	Let $h: Q\to\R^n$ be a $\gamma$-strongly convex function. Then its Fenchel conjugate $h^*$ is $1/\gamma$-smooth.
\end{proposition}

\subsection{Algorithms and complexities for dual smoothed problems}

Let us begin with problem \eqref{prob:consensus_reg_dual}. We will apply APAPC (Algorithm 1) in the modification of paper \cite{yarmoshik2024coupled}, where it was applied to coupled constraints optimization of strongly convex smooth functions.
\begin{theorem}\label{thm:consensus_reg_dual_convergence}
	Let $\bz^*$ be the solution of \eqref{prob:consensus_reg_dual} and let $\norm{\bz^0 - \bz^*}_2^2\leq R^2$. Also assume that regularizer $G$ is strongly convex. Then APAPC requires
	\begin{align*}
		N = O\cbraces{\sqrt{\frac{R^2}{\lambda\eps} \kappa_W \kappa_A} \ln\cbraces{\frac1\eps}}
	\end{align*}
	communication rounds in order to yield $\bz^N$ such that $\norm{\bz^N - \bz^*}_2^2\leq \eps$.
\end{theorem}
\begin{proof}
	By \Cref{prop:conjugate_smoothness}, function $F^*$ is $1/\lambda$-smooth. After that, adding a regularization term $\eps\norm{\bz}_2^2/(2R^2)$ we obtain a $(1/\lambda + \eps/R^2)$-smooth and $\eps/R^2$-strongly convex function. From \Cref{prop:apapc} we get the desired number of iterations.
\end{proof}

We now pass to problem \eqref{prob:coupled_reg_dual}.
\begin{theorem}\label{thm:coupled_reg_dual_convergence}
	Let $\bz^*$ be the solution of \eqref{prob:coupled_reg_dual} and let $\norm{\bz^0 - \bz^*}_2^2\leq R^2$. Also assume that regularizer $G$ is strongly convex. Then APAPC requires
	\begin{align*}
		N = O\cbraces{\sqrt{\frac{L_AR^2}{\lambda\eps} \kappa_W}} \ln\cbraces{\frac1\eps}
	\end{align*}
	communication rounds in order to yield $\bz^N$ such that $\norm{\bz^N - \bz^*}_2^2\leq \eps$.
\end{theorem}
\begin{proof}
	The proof is analogical to proof of \Cref{thm:consensus_reg_dual_convergence}.
\end{proof}

\subsection{Examples}

\noindent\textbf{Decentralized basis pursuit}.
Consider a special case of \eqref{prob:coupled_reg_dual} with $F(\bx) = \norm{\bx}_1,~ Q = \R^m,~ G(\bx) = \norm{\bx}_2^2 / 2$.
\begin{align*}
	\min_{x, y}~ &\norm{\bx}_1 + \frac\lambda 2 \norm{\bx}_2^2 \\
	\text{s.t. } &\mA\bx + \mW\by = \bb
\end{align*}
Fenchel conjugate $F_{\lambda G, Q}^*(\by)$ is computed as
\begin{align*}
	F_{\lambda G, Q}^*(\by) 
	&= \max_\bx \sbraces{\angles{\bx, \by} - \norm{\bx}_1 - \frac\lambda 2 \norm{\bx}_2^2} \\
	&= \frac{\norm{\by}_2^2}{2\lambda} - \lambda \proxv_{\norm{\cdot}_1/\lambda} \cbraces{\frac\by\lambda} \\
	&= \frac{\norm{\by}_2^2}{2\lambda} - \frac1\lambda \norm{\max\cbraces{|\by| - 1, 0}}_1 + \frac{1}{2\lambda} \norm{\by - \sign\by \odot \max\cbraces{|\by| - 1, 0}}_2^2.
\end{align*}
Problem \eqref{prob:coupled_reg_dual} takes the form
\begin{align*}
	\min_\bz~ &\Bigg[\frac{\norm{\mA^\top\bz}_2^2}{2} - \norm{\max\cbraces{|\mA^\top \bz| - 1, 0}}_1 \\
	&+ \frac{1}{2}\norm{\mA^\top\bz - \sign(\mA^\top\bz)\odot \max(|\mA^\top\bz| - 1, 0)}_2^2 - \lambda\angles{\bz, \bb} \Bigg] \\
	\text{s.t. } &\mW\bz = 0
\end{align*}

\vspace{0.3cm}
\noindent\textbf{Decentralized basis pursuit via double duality}. In the previous example, we regularized the decentralized basis pursuit problem and then took the dual. Now we first take the dual and then regularize it.
\begin{align*}
	\min_\bz~ &\angles{\bz, \bb} + \mathbb{I}(\norm{\mA^\top\bz}_\infty\leq 1) + \frac{\lambda}{2}\norm{\bz}_2^2 \\
	\text{s.t. } &\mW\bz = 0
\end{align*}
The problem above is an instance of \eqref{prob:consensus_reg_dual} with $F(\bz) = \angles{\bz, \bb} + \mathbb{I}(\norm{\mA^\top\bz}_\infty\leq 1)$ and $F(\bz) = \norm{\bz}_2^2/2$. Let us take its dual one more time. We get
\begin{align*}
	F_{\lambda G}^*(\bx) = \frac{\norm{\bx - \bb}_2^2}{2\lambda} - \lambda \sbraces{\one_m^\top \cdot \max\cbraces{\frac{|\bx - \bb|}{\lambda}, 0}}^2.
\end{align*}
Therefore, the second dual takes the form
\begin{align*}
	\min_\by~ &\frac{\norm{\mA^\top\by - \bb}_2^2}{2\lambda} - \lambda \sbraces{\one_m^\top \cdot \max\cbraces{\frac{|\mA^\top\by - \bb|}{\lambda}, 0}}^2 \\
	\text{s.t. } &\mW\bu + \mA^\top\by = 0
\end{align*}

\vspace{0.3cm}
\noindent\textbf{Decentralized mean squared error minimization}.
Consider an instance of consensus optimization with $F(\bx) = \norm{\bx}_1$ and $G(\bx) = \norm{\bx}_2^2/2$.
\begin{align*}
	\min_{\bx} &\norm{\by}_1 + \frac\lambda 2\norm{\by}_2^2 \\
	\text{s.t. } &\mW\bx = 0,~ \by = \mA\bx - \bb
\end{align*}
Its dual is an instance \eqref{prob:coupled_reg_dual} writes as
\begin{align*}
	\min_{\bz,\bu}~ &\frac{\norm{\bz}_2^2}{2} - \norm{\max\cbraces{|\bz| - 1, 0}}_1 + \frac{1}{2}\norm{\bz - \sign(\bz)\odot \max(|\bz| - 1, 0)}_2^2 + \lambda\angles{\bz, \bb} \\
	\text{s.t. } &\mW\bu + \mA^\top\bz = 0
\end{align*}

\vspace{0.3cm}
\noindent\textbf{Decentralized mean squared error minimization via double duality}.
Consider the problem similar to the previous example but first take its dual and then regularize. The dual with regularization writes as
\begin{align*}
	\min_{\bz\in\R^m}~ &\angles{\bz, \bb} + \mathbb{I}(\norm{\bz}_\infty\leq 1) + \frac{\lambda}{2}\norm{\bz}_2^2 \\
	\text{s.t. } &\mA^\top\bz + \mW\bu = 0
\end{align*}
The problem above has type \eqref{prob:coupled_reg} with $F(\bz) = \angles{\bz, \bb} + \mathbb{I}(\norm{\bz}_\infty\leq 1)$, $G(\bz) = \norm{\bz}_2^2/2$. Taking the dual for the second time, we obtain
\begin{align*}
	\min_\bx~ &\frac{\norm{\mA\bx - \bb}_2^2}{2\lambda} - \lambda\sbraces{\one_m^\top\cdot \max\cbraces{0, \frac{|\mA\bx - \bb|}{\lambda}}}^2 \\
	\text{s.t. } &\mW\bx = 0
\end{align*}

\section{Conclusion}\label{sec:conclusion}

In this paper we tried to systemize the dual approach to decentralized optimization. We considered two problem classes: consensus optimization of linear models and coupled constraints optimization. We showed that consensus and coupled constraints problems are dual to each other. We also deduced the dual problem formulations of regularized initial problems and analyzed how known decentralized optimization methods work on these problems. Finally, we illustrated our approach on decentralized basis pursuit problem and decentralized mean absolute error optimization.

\section*{Acknowledgments}\label{sec:acknowledgments}

The research was supported by Russian Science Foundation (project No. 23-11-00229), https://rscf.ru/en/project/23-11-00229/.

%\newpage
\bibliographystyle{abbrv}
\bibliography{references}
\end{document}